\documentclass[11pt,a4paper]{article}%
\usepackage{amsmath}
\usepackage{amsfonts}
\usepackage{amssymb}

\providecommand{\U}[1]{\protect\rule{.1in}{.1in}}
\providecommand{\U}[1]{\protect\rule{.1in}{.1in}}
\newtheorem{theorem}{Theorem}

\newtheorem{corollary}[theorem]{Corollary}

\newtheorem{definition}[theorem]{Definition}

\newtheorem{proposition}[theorem]{Proposition}
\newtheorem{remark}[theorem]{Remark}

\newenvironment{proof}[1][Proof]{\noindent\textbf{#1.} }{\ \rule{0.5em}{0.5em}}
\include {mak}
\title{\bf Some applications of the chromatic polynomials}
\author{Mohammed Said Maamra and Miloud Mihoubi\\[1.5ex]
Faculty of Mathematics, RECITS Laboratory, DG-RSDT\\
USTHB, BP 32, El-Alia, 16111, Algeries, Algeria\\
\small\tt mmaamra@usthb.dz \ \ \ mmihoubi@usthb.dz \\
}

\begin{document}

\maketitle

\begin{abstract}
The chromatic polynomials are studied by several authors and have important
applications in different frameworks, specially, in graph theory and enumerative combinatorics.
The aim of this work is to establish some properties of the coefficients of the chromatic polynomial
of a graph. Three applications on restricted Stirling numbers of the second kind are given.

  \bigskip\noindent \textbf{Keywords:} Chromatic polynomial of a graph; Stirling numbers; log-concavity; P\'{o}lya-frequency sequences.\\
  \small Mathematics Subject Classifications: 05C15; 05C69; 11B73
\end{abstract}

\section{Introduction}
The chromatic polynomial was introduced by Birkhoff \cite{bir1} and studied
later by Whitney \cite{whi1,whi2}, Birkhoff and Lewis \cite{bir2}, Read \cite%
{rea} and several other authors.\ The chromatic polynomial of a graph can be
used as a tool to find the number of possible partitions of a finite set
under some particular restraints such that the Stirling numbers of the
second kind \cite{whit}.
For a given graph $G=\left( V,E\right) $ of order $n$ and $\lambda \in
\mathbb{N}
,$ a mapping $f:V\rightarrow \left \{ 1,2,\ldots ,\lambda \right \} $ is
called a $\lambda $-coloring of $G$ if $f(u)\neq f(v)$ whenever the vertices
$u$ \ and $v$ are adjacent in $G$. The $\lambda $-colorings $f$ and $g$ of $%
G $ are regarded as distinct if $f(x)\neq g(x)$ for some $x$ in $G.$ The
chromatic polynomial $P(G,\lambda )$ counts the number of (proper) $\lambda $%
-colorings of $G.$ For example, it is known that $P\left( O_{n},\lambda
\right) =\lambda ^{n},$ $P(K_{n},\lambda )=\left( \lambda \right) _{n}$ and $%
P\left( T_{n},\lambda \right) =\lambda \left( \lambda -1\right) ^{n-1},$ $%
n\geq 1,$ where $O_{n}$ is a graph of order $n$ and without edges, $K_{n}$
is the complete graph of order $n,$ $T_{n}$ is a tree of order $n$ and $%
\left( \lambda \right) _{n}=\lambda \left( \lambda -1\right) \cdots \left(
\lambda -n+1\right) $ if $n\geq 1$ and $\left( \lambda \right) _{0}=1.$ More
generally, the chromatic polynomial of $G$ can be written as $P(G,\lambda
)=\sum \limits_{i=\chi \left( G\right) }^{n}\alpha _{i}\left( G\right)
\left( \lambda \right) _{i},$ see \cite[Thm. 1.4.1]{don}, where $\alpha
_{i}\left( G\right) $ is the number of ways of partitioning $V$ into $i$
independent sets and $\chi \left( G\right) $ is the chromatic number. For
use later, recall that, if $G_{1}=\left( V_{1},E_{1}\right) $ and $%
G_{2}=\left( V_{2},E_{2}\right) $ are graphs on disjoint sets of vertices,
their union is defined by the graph $G_{1}\cup G_{2}=\left( V_{1}\cup
V_{2},E_{1}\cup E_{2}\right) $ and $P(G_{1}\cup G_{2},\lambda
)=P(G_{1},\lambda )P(G_{2},\lambda ),$ see for example \cite[Sec.\ 1.2]{don}.
In this paper, for a given graph $H,$ we present some properties for the
families of graphs $O_{n}\cup H,$ $K_{n}\cup H$ and $T_{n}\cup H.$ In the
next section, we give some recurrence relations for the coefficients $\alpha
_{k}\left( O_{n}\cup H\right) ,$ $\alpha _{k}\left( K_{n}\cup H\right) $ and
$\alpha _{k}\left( T_{n}\cup H\right) $ and some results on log-concavity and P\'{o}lya-frequency for sequences  related to these coefficients. In the three last sections, we present three applications on restricted Stirling numbers of the second kind.

\section{Recurrence relations and some consequences}
Let $H$ be any graph of $h$ vertices, $O_{n}$ be the graph of $n$ $\left(
\geq 1\right) $ vertices and no edges and let $O_{0}$ be the graph with no
vertices, $K_{n}$ be the complete graph of $n\  \left( \geq 1\right) $
vertices with $K_{0}$ be a graph with no vertices and $T_{n}$ be a tree of $%
n\  \left( \geq 1\right) $ vertices with $T_{0}$ be a graph with no vertices.
In this section, we give some recurrence relations for the coefficients $%
\alpha _{k}\left( O_{n}\cup H\right) ,$ $\alpha _{k}\left( K_{n}\cup
H\right) ,$ $\alpha _{k}\left( T_{n}\cup H\right) $ and some of their consequences.
\begin{theorem}
\label{TT1}Let $n,$ $s,$ $k$ be nonnegative integers with $\ 0\leq s\leq n.$
Then, $\alpha _{k}\left( O_{n}\cup H\right) =0$ if $k<\chi \left( H\right) $
or $k>n+h$ and for $\chi \left( H\right) \leq k\leq n+h$\ we have%
\begin{equation*}
\alpha _{k}\left( O_{n}\cup H\right) =\sum_{j=\chi \left( H\right)
}^{k}{s+j \brace k }_{\!\!j}\alpha _{j}\left( O_{n-s}\cup H\right) .
\end{equation*}%
In particular, for $s=n,$ we get%
\begin{equation*}
\alpha _{k}\left( O_{n}\cup H\right) =\sum_{j=\chi \left( H\right)
}^{k}{n+j \brace k }_{\!\!j}\alpha _{j}\left( H\right) ,
\end{equation*}%
and, for $s=1,$ we get%
\begin{equation*}
\alpha _{k}\left( O_{n}\cup H\right) =k\alpha _{k}\left( O_{n-1}\cup
H\right) +\alpha _{k-1}\left( O_{n-1}\cup H\right) ,\  \ n\geq 1,
\end{equation*}%
where the numbers ${n \brace k }_{\!\!r}$ are the $r$-Stirling numbers of
the second kind.
\end{theorem}

\begin{proof}
From \cite[Sec. 1.2]{don} we have $P\left( O_{n}\cup H,\lambda \right)
=\lambda ^{s}P\left( O_{n-s}\cup H,\lambda \right) ,\  \  \ 0\leq s\leq n.$%
\newline
Then, since the graph $O_{n}\cup H$ is of order $n+h$ and%
\begin{equation*}
\chi \left( O_{n}\cup H\right) =\max \left( \chi \left( O_{n}\right) ,\chi
\left( H\right) \right) =\max \left( 1,\chi \left( H\right) \right) =\chi
\left( H\right) ,
\end{equation*}%
we may state that $\alpha _{k}\left( O_{n}\cup H\right) =0$ if $k<\chi
\left( H\right) $ or $k>n+h,$ and otherwise, we have%
\begin{equation*}
\underset{k=\chi \left( H\right) }{\overset{n+h}{\sum }}\alpha _{k}\left(
O_{n}\cup H\right) \left( \lambda \right) _{k}=\underset{j=\chi \left(
H\right) }{\overset{n+h}{\sum }}\alpha _{j}\left( O_{n-s}\cup H\right)
\lambda ^{s}\left( \lambda \right) _{j}.
\end{equation*}%
On using the known identity $\left( z+j\right) ^{s}=\sum_{k=0}^{s}{n+j \brace k+j }_{\!\!j}\left( z\right) _{k},$
see \cite{bro}, we can write%
\begin{equation*}
\lambda ^{s}\left( \lambda \right) _{j}=\left( \lambda \right) _{j}\left(
\lambda -j+j\right) ^{s}=\overset{s}{\underset{k=0}{\sum }}{n+j \brace k+j }_{\!\!j}
\left( \lambda \right) _{j}\left( \lambda -j\right) _{k}
\end{equation*}%
and since $\left( \lambda \right) _{j}\left( \lambda -j\right) _{k}=\left(
\lambda \right) _{k+j}$ we get $\lambda ^{s}\left( \lambda \right) _{j}=%
\overset{s}{\underset{k=0}{\sum }}{n+j \brace k+j }_{\!\!j}\left( \lambda
\right) _{k+j}=\overset{s+j}{\underset{k=j}{\sum }}{s+j \brace k }_{\!\!j}\left( \lambda \right) _{k}.$ Then%
\begin{eqnarray*}
\underset{k=\chi \left( H\right) }{\overset{n+h}{\sum }}\alpha _{k}\left(
O_{n}\cup H\right) \left( \lambda \right) _{k} &=&\underset{j=\chi \left(
H\right) }{\overset{n+h}{\sum }}\alpha _{j}\left( O_{n-s}\cup H\right)
\overset{s+j}{\underset{k=j}{\sum }}{s+j \brace k }_{\!\!j}\left( \lambda
\right) _{k} \\
&=&\overset{s+j}{\underset{k=j}{\sum }}\left( \lambda \right)
_{k}\sum_{j=\chi \left( H\right) }^{k}\alpha _{j}\left( O_{n-s}\cup H\right)
{s+j \brace k }_{\!\!j}.
\end{eqnarray*}%
So, we get $\alpha _{k}\left( O_{n}\cup H\right) =\sum_{j=\chi \left(
H\right) }^{k}\alpha _{j}\left( O_{n-s}\cup H\right) {s+j \brace k }_{\!\!j}.$
\end{proof}

\begin{theorem}
\label{TT2}Let $n,$ $s,$ $k$ be nonnegative integers with$\ 0\leq s\leq n.$
Then, $\alpha _{k}\left( K_{n}\cup H\right) =0$ if $k<\max \left( n,\chi
\left( H\right) \right) $ or $k>n+h$ and for $\max \left( n,\chi \left(
H\right) \right) \leq k\leq n+h$\ we have%
\begin{equation*}
\alpha _{k}\left( K_{n}\cup H\right) =\underset{j=\max \left( n-s,\chi \left(
H\right) \right) }{\overset{k}{\sum }}\binom{s}{k-j}\left( j+s-n\right)
_{s+j-k}\alpha _{j}\left( K_{n-s}\cup H\right) .
\end{equation*}%
In particular, for $s=n,$ we get%
\begin{equation*}
\alpha _{k}\left( K_{n}\cup H\right) =\underset{j=\chi \left(
H\right)  }{\overset{k}{\sum }}\binom{n}{k-j}\left( j\right)
_{n+j-k}\alpha _{j}\left( H\right) ,
\end{equation*}%
and, for $s=1,$ we get%
\begin{equation*}
\alpha _{k}\left( K_{n}\cup H\right) =\left( k-n+1\right) \alpha _{k}\left(
K_{n-1}\cup H\right) +\alpha _{k-1}\left( K_{n-1}\cup H\right) ,\  \ n\geq 1.
\end{equation*}
\end{theorem}

\begin{proof}
From \cite[Sec. 1.2]{don} we have $P\left( K_{n}\cup H,\lambda \right)
=\left( \lambda -n+1\right) P\left( K_{n-1}\cup H,\lambda \right) .$ Hence
\begin{equation*}
P\left( K_{n}\cup H,\lambda \right) =\left( \lambda -n+s\right) _{s}P\left(
K_{n-s}\cup H,\lambda \right) ,\  \ 0\leq s\leq n.
\end{equation*}%
Then, since the graph $K_{n}\cup H$ is of order $n+h$ and%
\begin{equation*}
\chi \left( K_{n}\cup H\right) =\max \left( \chi \left( K_{n}\right) ,\chi
\left( H\right) \right) =\max \left( n,\chi \left( H\right) \right) ,
\end{equation*}%
we state that $\alpha _{k}\left( K_{n}\cup H\right) =0$ if $k<\max \left(
n,\chi \left( H\right) \right) $ or $k>n+h,$ and otherwise, we have%
\begin{equation*}
\underset{k=\max \left( n,\chi \left( H\right) \right) }{\overset{n+h}{\sum }%
}\alpha _{k}\left( K_{n}\cup H\right) \left( \lambda \right) _{k}=\underset{%
j=\max \left( n-s,\chi \left( H\right) \right) }{\overset{n-s+h}{\sum }}\alpha
_{j}\left( K_{n-s}\cup H\right) \left( \lambda -n+s\right) _{s}\left(
\lambda \right) _{j}.
\end{equation*}%
By the fact that the sequence of polynomials $((\lambda)_n, n\geq0)$ is of binomial type, the expression $\left( \lambda -n+s\right) _{s}\left(\lambda \right) _{j}$ can be written as
\begin{eqnarray*}
\left( \lambda -n+s\right) _{s}\left( \lambda \right) _{j} &=&\left( \lambda
-j+j+s-n\right) _{s}\left( \lambda \right) _{j} \\
&=&\underset{k=0}{\overset{s}{\sum }}\binom{s}{k}\left( j+s-n\right)
_{s-k}\left( \lambda -j\right) _{k}\left( \lambda \right) _{j} \\
&=&\underset{k=0}{\overset{s}{\sum }}\binom{s}{k}\left( j+s-n\right)
_{s-k}\left( \lambda \right) _{k+j} \\
&=&\underset{k=j}{\overset{s+j}{\sum }}\binom{s}{k-j}\left( j+s-n\right)
_{s+j-k}\left( \lambda \right) _{k}.
\end{eqnarray*}%
So, it results that
\begin{eqnarray*}
&&\underset{k=\max \left( n,\chi \left( H\right) \right) }{\overset{n+h}{%
\sum }}\alpha _{k}\left( K_{n}\cup H\right) \left( \lambda \right) _{k} \\
&=&\underset{j=\max \left( n-s,\chi \left( H\right) \right) }{\overset{n-s+h}{%
\sum }}\alpha _{j}\left( K_{n-s}\cup H\right) \underset{k=j}{\overset{s+j}{%
\sum }}\binom{s}{k-j}\left( j+s-n\right) _{s+j-k}\left( \lambda \right) _{k}
\\
&=&\underset{k=\max \left( n,\chi \left( H\right) \right) }{\overset{n+h}{%
\sum }}\left( \lambda \right) _{k}\underset{j=\max \left( n-s,\chi \left(
H\right) \right) }{\overset{k}{\sum }}\binom{s}{k-j}\left( j+s-n\right)
_{s+j-k}\alpha _{j}\left( K_{n-s}\cup H\right).
\end{eqnarray*}%
The identification of the coefficients of $\left( \lambda \right) _{k}$ complete the proof.
\end{proof}

\begin{theorem}
\label{TT3}Let $n,$ $s,$ $k$ be nonnegative integers with$\ 0\leq s\leq n.$
Then
\begin{equation*}
\alpha _{k}\left( T_{n}\cup H\right) =0\text{ \ for \ }k<\max \left(
2-\delta _{\left( n=1\right) }-2\delta _{\left( n=0\right) },\chi \left(
H\right) \right) \text{ \ or \ }k>n+h,
\end{equation*}%
and, for $\max \left( 2-\delta _{\left( n=1\right) }-2\delta _{\left(
n=0\right) },\chi \left( H\right) \right) \leq k\leq n+h$\ we have%
\begin{equation*}
\alpha _{k}\left( T_{n}\cup H\right) =\underset{j=\max \left( k-s,0\right) }{%
\overset{k}{\sum }}{s+j-1 \brace k-1 }_{\!\!j-1}\alpha _{j}\left(
T_{n-s}\cup H\right) .
\end{equation*}%
In particular, for $s=n,$ we get%
\begin{equation*}
\alpha _{k}\left( T_{n}\cup H\right) =\underset{j=1}{\overset{k}{\sum }}%
{s+j-1 \brace k-1 }_{\!\!j-1}\alpha _{j}\left( H\right)
\end{equation*}%
and, for $s=1,$ we get%
\begin{equation*}
\alpha _{k}\left( T_{n}\cup H\right) =\left(k-1 \right)\alpha _{k}\left( T_{n-1}\cup
H\right) +\alpha _{k-1}\left( T_{n-1}\cup H\right) ,\  \ n\geq 1.
\end{equation*}
\end{theorem}

\begin{proof}
The graph $T_{n}\cup H$ is of order $n+h$ and we have%
\begin{equation*}
\chi \left( T_{n}\cup H\right) =\max \left( \chi \left( T_{n}\right) ,\chi
\left( H\right) \right) =\max \left( 2-\delta _{\left( n=1\right) }-2\delta
_{\left( n=0\right) },\chi \left( H\right) \right) ,
\end{equation*}%
we conclude that $\alpha _{k}\left( T_{n}\cup H\right) =0$ if $k<\max \left(
2-\delta _{\left( n=1\right) }-2\delta _{\left( n=0\right) },\chi \left(
H\right) \right) $ or $k>n+h,$ and otherwise, we have%
\begin{equation*}
P\left( T_{n}\cup H,\lambda \right) =P\left( T_{n},\lambda \right) P\left(
H,\lambda \right) =\lambda \left( \lambda -1\right) ^{n-1}P\left( H,\lambda
\right) ,
\end{equation*}%
which gives $P\left( T_{n}\cup H,\lambda \right) =\left( \lambda -1\right)
^{s}P\left( T_{n-s}\cup H,\lambda \right) ,\  \ 0\leq s\leq n.$ \newline
This relation is equivalent to%
\begin{equation*}
\underset{k=0}{\overset{n+h}{\sum }}\alpha _{k}\left( T_{n}\cup H\right)
\left( \lambda \right) _{k}=\underset{j=0}{\overset{n+h-s}{\sum }}\alpha
_{j}\left( T_{n-s}\cup H\right) \left( \lambda -1\right) ^{s}\left( \lambda
\right) _{j}.
\end{equation*}%
Similarly to $\lambda ^{s}\left( \lambda \right) _{j}$ (see proof of
Theorem \ref{TT1}), we get
\begin{equation*}
\left( \lambda -1\right) ^{s}\left( \lambda \right) _{j}=\underset{k=j}{%
\overset{s+j}{\sum }}{s+j-1 \brace k-1 }_{\!\!j-1}\left( \lambda \right)
_{k}.
\end{equation*}
Then, the last equality becomes
\begin{eqnarray*}
\underset{k=0}{\overset{n+h}{\sum }}\alpha _{k}\left( T_{n}\cup H\right)
\left( \lambda \right) _{k} &=&\underset{j=0}{\overset{n+h-s}{\sum }}\alpha
_{j}\left( T_{n-s}\cup H\right) \underset{k=j}{\overset{s+j}{\sum }}{s+j-1 \brace k-1 }_{\!\!j-1}\left( \lambda \right) _{k} \\
&=&\underset{k=0}{\overset{n+h}{\sum }}\left( \lambda \right) _{k}\underset{%
j=\max \left( k-s,0\right) }{\overset{k}{\sum }}{s+j-1 \brace k-1 }_{\!\!j-1}\alpha _{j}\left( T_{n-s}\cup H\right) ,
\end{eqnarray*}%
which gives thus the result.
\end{proof}

\noindent To give the following proposition, let us recall some definitions
and results on log-concavity, P\'{o}lya-frequency and $q$-log-convexity.
Indeed, let $u_{0},$ $u_{1},$ $u_{2},\ldots ,$ be a sequence of nonnegative
real numbers. It is log-concave (LC) if $u_{i-1}u_{i+1}\leq u_{i}^{2}$ for
all $i>0,$ and, it is called a P\'{o}lya-frequency sequence (or a PF
sequence) if all minors of the matrix $A=(u_{i-j})_{i,j\geq 0}$ have
nonnegative determinants (where $u_{k}=0$ if $k<0$), for more information
see \cite{kar}. A sequence of real polynomials $\left( P_{n}(q),n\geq
0\right) $ is called $q$-log-convex if the polynomial $%
P_{n}(q)^{2}-P_{n-1}(q)P_{n+1}(q)$ has nonnegative coefficients for all $%
n\geq 1,$ see \cite{sag1,sag2,zho}. Some known results on such sequences are
given as follows. \newline
Let $\left( T\left( n,k\right) ,\ n,k\geq 0\right) $ be sequence of
nonnegative numbers satisfying the recurrence
\begin{equation*}
T\left( n,k\right) =\left( a_{1}n+a_{2}k+a_{3}\right) T\left( n-1,k\right)
+\left( b_{1}n+b_{2}k+b_{3}\right) T\left( n-1,k-1\right), \ n\geq k\geq 1,
\end{equation*}%
with $T\left( n,k\right) =0$ unless $0\leq k\leq n,$ $T\left(0,0\right)>0,$ \ $a_{1}\geq
0,$ $a_{1}+a_{2}\geq 0,\ a_{1}+a_{3}\geq 0$ and $b_{1}\geq 0,$ $%
b_{1}+b_{2}\geq 0,$ $b_{1}+b_{2}+b_{3}\geq 0.$ It is shown in \cite[Thm. 2]%
{kur} that, for each fixed $n,$ the sequence $\left( T\left( n,k\right) ,\
0\leq k\leq n\right) $ is log-concave. If further we have $a_{2}b_{1}\geq
a_{1}b_{2}$ \ and \ $a_{2}\left( b_{1}+b_{2}+b_{3}\right) \geq \left(
a_{1}+a_{3}\right) b_{2},$ this sequence is P\'{o}lya frequency sequence
\cite[Cor. 3]{wan} and further, by setting $T_{n}\left( q\right) =\overset{n}%
{\underset{k=0}{\sum }}T\left( n,k\right) q^{k},$ if
\begin{equation*}
\left( a_{2}b_{1}-a_{1}b_{2}\right) n+a_{2}b_{2}k+a_{2}b_{3}-a_{3}b_{2}\geq 0%
\text{ for }0<k\leq n,
\end{equation*}%
then, the sequence of polynomials $\left( T_{n}\left( q\right) ,\ n\geq
0\right) $ is $q$-log-convex \cite[Thm. 4.1]{liu}.

\begin{proposition}
\label{P1}Let $\left( U\left( n,k\right) ,\ n,k\geq 0\right) ,$  $\left( V\left( n,k\right) ,\ n,k\geq 0\right) ,$ and $\left( W\left( n,k\right) ,\ n,k\geq 0\right) ,$ be sequences of nonnegative numbers with
$U\left( n,k\right)= V\left( n,k\right)= W\left( n,k\right)=0$ when  $ k> n$ and for $0\leq k\leq n,$ these sequences are defined by
\begin{eqnarray*}
U\left( n,k\right) & =&\alpha _{k+h}\left( O_{n}\cup H\right) ,\\
V\left( n,k\right)  &=&\alpha _{k+h}\left( T_{n}\cup H\right) ,\\
W\left( n,k\right) & =&\alpha _{k+h}\left( K_{n}\cup H\right)
\end{eqnarray*}%
and let
\begin{eqnarray*}
U_{n}\left( q\right) &=&\overset{n}{\underset{k=0}{\sum }}U\left( n,k\right) q^{k}, \\
V_{n}\left( q\right) &=&\overset{n}{\underset{k=0}{\sum }}V\left( n,k\right) q^{k}.
\end{eqnarray*}%
Then, the sequences $\left( U\left( n,k\right) ,\ 0\leq k\leq n\right) $ and $%
\left( V\left( n,k\right) ,\ 0\leq k\leq n\right) $ are log-concave and P%
\'{o}lya frequency sequences, the sequences of real polynomials $\left(
U_{n}(q),n\geq 0\right) $ and $\left( V_{n}(q),n\geq 0\right) $ are $q$%
-log-convex sequences and the sequence $\left( W\left( n,k\right) ,\ 0\leq
k\leq n\right) $ is a P\'{o}lya frequency sequence.
\end{proposition}

\begin{proof}
We have $U\left(0,0\right) =V\left(0,0\right) = W\left(0,0\right) =\alpha _{h}\left(H\right)=1 $ and for $n\geq 1$, Theorems \ref{TT1}, \ref{TT2} and \ref{TT3} imply
\begin{eqnarray*}
U\left(n,k\right) &=&U\left( n-1,k-1\right) +\left( k+h\right) U\left( n-1,k\right) ,\\
V\left(n,k\right) &=&V\left( n-1,k-1\right) +\left( k+h-1\right) V\left( n-1,k\right) , \\
W\left(n,k\right) &=&W\left( n-1,k-1\right) +\left( k-n+h+1\right) W\left(n-1,k\right) .
\end{eqnarray*}%
So, the log-concavity follows from \cite[Thm. 2]{kur}, P\'{o}lya frequency
follows from \cite[Cor. 3]{wan} and $q$-log-convexity follows from \cite[%
Thm. 4.1]{liu}.
\end{proof}

\section{Application to the graph $K_{r_{_{1}}}\cup \cdots \cup K_{r_{_{p}}}\cup O_{n}$}

Let $p\geq 1$ be an integer. We consider the graph $G_{n,\mathbf{r}%
_{p}}=K_{r_{_{1}}}\cup \cdots \cup K_{r_{_{p}}}\cup O_{n}$ of order $%
n+r_{1}+\cdots +r_{p}$ and chromatic number
\begin{equation*}
\chi \left( G_{n,\mathbf{r}_{p}}\right) =\max \left( \chi \left(
K_{r_{_{1}}}\right) ,\ldots ,\chi \left( K_{r_{_{p}}}\right) ,\chi \left(
O_{n}\right) \right) =\max \left( r_{1},\ldots ,r_{p},1\right) .
\end{equation*}%
First of all, recall the definition of the $\left( r_{1},\ldots
,r_{p}\right) $-Stirling number of the second kind introduced by Mihoubi et al. in \cite{mih,maa} .

\begin{definition}
Let $R_{1},\ldots ,R_{p}$ be subsets of the set $\left[ n\right] $ with $%
\left \vert R_{i}\right \vert =r_{i}$ and $R_{i}\cap R_{j}=\varnothing $ for
all $i,j=1,\ldots ,p,$ $i\neq j.$ The $\left( r_{1},\ldots ,r_{p}\right) $%
-Stirling number of the second kind, $p\geq 1,$ denoted by ${ n \brace k }_{\!\!r_{1},\ldots,r_{p}},$
counts the number of partitions of the set $%
\left[ n\right] :=\left \{ 1,2,\ldots ,n\right \} $ into $k$ non-empty
subsets such that the elements of each of the $p$ sets $R_{1},\ldots ,R_{p}$
are in distinct subsets.
\end{definition}

\noindent From this definition, one can see easily that these numbers satisfy%
\begin{align*}
{ n \brace k }_{\!\!r_{1},\ldots,r_{p}}& =0,\text{ \  \  \ }n<r_{1}+\cdots
+r_{p}\text{ or }k<\max \left( r_{1},\ldots ,r_{p},1\right) , \\
{ n \brace k }_{\!\!r_{1},\ldots,r_{p}}& ={n\brace k}_{\!\!r_{p}}%
\text{ if }r_{1},\ldots ,r_{p-1}\in \left \{ 0,1\right \} , \\
{ n \brace k }_{\!\!r_{1},\ldots,r_{p}}& ={ n \brace k }_{\!\!r_{\sigma
\left( 1\right) },\ldots ,r_{\sigma \left( p\right) }}\text{ for all
permutations }\sigma \text{ on the set }\left \{ 1,\ldots ,p\right \} .
\end{align*}%
Furthermore, for $r_{1}\leq r_{2}\leq \cdots \leq r_{p},$ the coefficient $%
\alpha _{k}\left( G_{n,\mathbf{r}_{p}}\right) $ represents the number of
ways of partitioning the set of $%
n+\left \vert \mathbf{r}_{p}\right \vert $ vertices of $G_{n,\mathbf{r}%
_{p}} $ into $k$ independent subsets, and by the definition of the graph $%
G_{n,\mathbf{r}_{p}},$ the elements of each of the $p$ subgraphs $%
K_{r_{_{1}}},\ldots ,K_{r_{_{p}}}$ must be in distinct subsets. So, this
number is exactly ${n+\left \vert \mathbf{r}_{p}\right \vert \brace k}_{\!\!\mathbf{r}_{p}}.$
Then, we may state that%
\begin{equation*}
\alpha _{k}\left( G_{n,\mathbf{r}_{p}}\right) ={n+\left \vert \mathbf{r}_{p}\right \vert \brace k}_{\!\!\mathbf{r}_{p}} .
\end{equation*}%
In particular%
\begin{equation*}
\alpha _{k}\left( K_{r}\cup O_{n}\right) ={n+r \brace k}_{\!\!r} \text{ \
and \ }\alpha _{k}\left( O_{n}\right) ={n\brace k}.
\end{equation*}%
Now, we present new proofs for some properties of the $\left( r_{1},\ldots
,r_{p}\right) $-Stirling numbers of the second kind and other
results.

\begin{theorem}
\label{T1}For $r_{1}\leq r_{2}\leq \cdots \leq r_{p},$ the polynomial $%
\left( z+r_{p}\right) _{r_{1}}\cdots \left( z+r_{p}\right) _{r_{p-1}}\left(
z+r_{p}\right) ^{n}$ can be written in the basis $\left \{ \left( \lambda
\right) _{k},\ k=0,1,\ldots,n+\left \vert \mathbf{r}_{p-1}\right \vert \right \} $ as follows%
\begin{equation*}
\left( z+r_{p}\right) _{r_{1}}\cdots \left( z+r_{p}\right) _{r_{p-1}}\left(
z+r_{p}\right) ^{n}=\sum_{k=0}^{n+\left \vert \mathbf{r}_{p-1}\right \vert
}{n+\left \vert \mathbf{r}_{p}\right \vert \brace k}_{\!\!\mathbf{r}_{p}}\left( z\right) _{k},
\end{equation*}%
where $\mathbf{r}_{p}:=\left( r_{1},\ldots ,r_{p}\right) $ and $\left \vert
\mathbf{r}_{p}\right \vert :=r_{1}+\cdots +r_{p}.$ In particular, we have
\begin{equation*}
\left( z+r\right) ^{n}=\sum_{k=0}^{n}{n+r \brace k+r}_{\!\!r}\left(
z\right) _{k}.
\end{equation*}
\end{theorem}

\begin{proof}
From \cite[Sec. 1.2]{don}, the chromatic polynomial of the graph $G_{n,%
\mathbf{r}_{p}}$ is%
\begin{eqnarray*}
P\left( G_{n,\mathbf{r}_{p}},z+r_{p}\right) &=&P\left( O_{n},z+r_{p}\right)
P\left( K_{r_{1}},z+r_{p}\right) \cdots P\left( K_{r_{p}},z+r_{p}\right) \\
&=&\left( z+r_{p}\right) ^{n}\left( z+r_{p}\right) _{r_{1}}\left(
z+r_{p}\right) _{r_{2}}\cdots \left( z+r_{p}\right) _{r_{p}},
\end{eqnarray*}%
and, by definition of the chromatic polynomial, we have%
\begin{eqnarray*}
P\left( G_{n,\mathbf{r}_{p}},z+r_{p}\right)& =&\sum_{k=0}^{n+\left \vert
\mathbf{r}_{p-1}\right \vert }{n+\left \vert \mathbf{r}%
_{p}\right \vert \brace k+r_{p}}_{\!\!\mathbf{r}_{p}}\left( z+r_{p}\right)
_{k+r_{p}} \\ &=&\sum_{k=0}^{n+\left \vert \mathbf{r}_{p-1}\right \vert }
{n+\left \vert \mathbf{r}_{p}\right \vert \brace k+r_{p}}_{\!\!\mathbf{r}%
_{p}}\left( z+r_{p}\right) _{r_{p}}\left( z\right) _{k}.
\end{eqnarray*}%
which complete the proof.
\end{proof}

\begin{remark}
\label{R1}From the following identity (see \cite{tom} or \cite[pp. 102]{don}%
) we have%
\begin{equation*}
\alpha _{k}\left( G_{n,\mathbf{r}_{p}}\right) =\frac{1}{k!}%
\sum_{j=0}^{k}\left( -1\right) ^{k-j}\binom{k}{j}P\left( G_{n,\mathbf{r}%
_{p}},j\right) ,
\end{equation*}%
or equivalently,%
\begin{eqnarray*}
{n+\left \vert \mathbf{r}_{p}\right \vert \brace k+r_{p}}_{\!\!\mathbf{r%
}_{p}} &=&\frac{1}{\left( k+r_{p}\right) !}\sum_{j=r_{p}}^{k+r_{p}}\left(
-1\right) ^{k+r_{p}-j}\binom{k+r_{p}}{j}\left( j\right) _{r_{1}}\cdots
\left( j\right) _{r_{p}}j^{n} \\
&=&\frac{1}{\left( k+r_{p}\right) !}\sum_{j=0}^{k}\left( -1\right) ^{k-j}%
\binom{k+r_{p}}{j+r_{p}}\left( j+r_{p}\right) _{r_{1}}\cdots \left(
j+r_{p}\right) _{r_{p}}\left( j+r_{p}\right) ^{n} \\
&=&\frac{1}{k!}\sum_{j=0}^{k}\left( -1\right) ^{k-j}\binom{k}{j}\left(
j+r_{p}\right) _{r_{1}}\cdots \left( j+r_{p}\right) _{r_{p-1}}\left(
j+r_{p}\right) ^{n}.
\end{eqnarray*}%
So, the two initial values of these numbers are given by%
\begin{eqnarray*}
{n+\left \vert \mathbf{r}_{p}\right \vert \brace r_{p}}_{\!\!\mathbf{r}%
_{p}} &=&\left( r_{p}\right) ^{n}\left( r_{p}\right) _{r_{1}}\cdots \left(
r_{p}\right) _{r_{p-1}}, \\
{n+\left \vert \mathbf{r}_{p}\right \vert \brace r_{p}+1}_{\!\!\mathbf{r%
}_{p}} &=&\left( r_{p}+1\right) _{r_{1}}\cdots \left( r_{p}+1\right)
_{r_{p-1}}\left( r_{p}+1\right) ^{n}-\left( r_{p}\right) _{r_{1}}\cdots
\left( r_{p}\right) _{r_{p-1}}\left( r_{p}\right) ^{n}.
\end{eqnarray*}
\end{remark}

\noindent Now, set $K_{r_{_{i}}}\cup H:=K_{r_{_{1}}}\cup \cdots \cup K_{r_{_{p}}}\cup O_{n} \ \left(i=1,\ldots ,p\right)$ in Theorem \ref{TT2} to obtain:

\begin{corollary}
Let $n, \ s$ and $k$ be nonnegative integers with $n\geq \left \vert \mathbf{r}_{p}\right \vert $.
Then, for a fixed $ i$ in the set $\left\{1,\ldots ,p\right\}$,
if $r_{i}\geq 1$  we have%
\begin{equation*}
{n \brace k}_{\!\!\mathbf{r}_{p}}={n-1 \brace k-1}_{\!\!\mathbf{r}%
_{p}-e_{i}}+\left( k+1-r_{i}\right) {n-1 \brace k}_{\!\!\mathbf{r}%
_{p}-e_{i}},
\end{equation*}%
where $\mathbf{e}_{i}$ denote the i-$th$ vector of the canonical basis of $%
\mathbb{R}^{p}.$
\end{corollary}

\noindent A particular case of the last corollary is when $p=1$ we obtain
the following known identity%
\begin{equation*}
{n \brace k}_{\!\!r}={n-1 \brace k-1}_{\!\!r-1}+\left( k+1-r\right){n-1 \brace k}_{\!\!r-1}.
\end{equation*}

\noindent By setting $ H:=K_{r_{_{1}}}\cup \cdots \cup K_{r_{_{p}}}$ in Theorem \ref{TT1} we obtain:

\begin{corollary}
\label{C1}Let $n,\ k$ be integers such that $r_{p}\leq k\leq n$ and $n\geq
\left \vert \mathbf{r}_{p}\right \vert .$ We have%
\begin{equation*}
{n \brace k}_{\!\!\mathbf{r}_{p}}=\sum_{j=r_{p}}^{k}{s+j \brace k}_{\!\!j}
{n-s \brace j}_{\!\!\mathbf{r}_{p}},\  \ 0\leq s\leq n-\left \vert \mathbf{r}_{p}\right \vert .
\end{equation*}%
In particular, for $s=1,$ these numbers obey to the following recurrence relation:%
\begin{equation*}
{n \brace k}_{\!\!\mathbf{r}_{p}}={n-1 \brace k-1}_{\!\!\mathbf{r}%
_{p}}+k{n-1 \brace k}_{\!\!\mathbf{r}_{p}},
\end{equation*}%
and for $s=n-\left \vert \mathbf{r}_{p}\right \vert $ we get%
\begin{equation*}
{n \brace k}_{\!\!\mathbf{r}_{p}}=\sum_{j=r_{p}}^{k}
{n-\left \vert \mathbf{r}_{p}\right \vert +j \brace k}_{\!\!j}{\left
\vert \mathbf{r}_{p}\right \vert  \brace j}_{\!\!\mathbf{r}_{p}}.
\end{equation*}
\end{corollary}

\noindent In particular, for $p=s=1$ in Corollary \ref{C1}, we obtain the known identity%
\begin{equation*}
{n \brace k}_{\!\!r}={n-1 \brace k-1}_{\!\!r}+k{n-1 \brace k}_{\!\!r},
\end{equation*}%
and for $s=n-\left \vert \mathbf{r}_{p}\right \vert ,$ $p=2$ and $\left(
r_{1},r_{2}\right) =\left( 1,r\right) $ in Corollary \ref{C1}, we obtain
\begin{equation*}
{n \brace k}_{\!\!r}={n \brace k}_{\!\!1,r}=\left( r+1\right)
{n-1 \brace k}_{\!\!r}+{n \brace k}_{\!\!r+1}.
\end{equation*}%
Now, upon using Proposition \ref{P1}, we may state the following corollary.
\begin{corollary}
\label{C3}For $0\leq k\leq n$ let
\begin{equation*}
U\left( n,k\right) ={n+\left \vert \mathbf{r}_{p}\right \vert
\brace k+\left \vert \mathbf{r}_{p}\right \vert }_{\!\!\mathbf{r}_{p}}.
\end{equation*}%
Then, the sequence $\left( U\left( n,k\right) ,\ 0\leq k\leq n\right) $ is
log-concave and P\'{o}lya frequency sequence, and, the sequence of
polynomials $\left( U_{n}\left( q\right) ,n\geq 0\right) $ defined by
\begin{equation*}
U_{n}\left( q\right) =\overset{n}{\underset{k=0}{\sum }}{n+\left \vert \mathbf{r}_{p}\right \vert
\brace k+\left \vert \mathbf{r}_{p}\right \vert }_{\!\!\mathbf{r}_{p}}q^{k}
\end{equation*}%
is a $q$-log-convex sequence.
\end{corollary}

\section{Application to the graph $K_{r_{1},\ldots,r_{p}}\cup O_{n}$}

Let $p\geq 2$ be an integer. For a second application of the chromatic
polynomials, we consider in this section the graph $K_{n,\mathbf{r}%
_{p}}=K_{r_{1},\ldots ,r_{p}}\cup O_{n}$ of
order $n+r_{1}+\cdots +r_{p}$ and chromatic number%
\begin{equation*}
\chi \left( K_{n,\mathbf{r}_{p}}\right) =\max \left( \chi \left(
K_{r_{1},\ldots ,r_{p}}\right) ,\chi \left( O_{n}\right) \right) =\max
\left( p,1\right)=p .
\end{equation*}%
Similarly to the numbers $\left( r_{1},\ldots ,r_{p}\right) $-Stirling
numbers, we present here a new
class of the Stirling numbers having also triangular recurrence relation.%
\newline
To start, let us giving the following definition.

\begin{definition}
Let $R_{1},\ldots ,R_{p}$ be subsets of the set $\left[ n\right] $ with $%
\left \vert R_{i}\right \vert =r_{i}$ and $R_{i}\cap R_{j}=\varnothing $ for
all $i,j=1,\ldots ,p,$ $i\neq j.$ The $K\left( r_{1},\ldots ,r_{p}\right) $%
-Stirling number of the second kind, denoted by $
{n \brace k}_{\!\!K\left( r_{1},\ldots ,r_{p}\right) },$ counts the number of
partitions of the set $\left[ n\right] $ into $k$ non-empty subsets such
that if $x\in R_{i}$ and $y\in R_{j}$ with $i\neq j,$ then each subset do
not containing simultaneously $x$ and $y$.
\end{definition}

\noindent From this definition, we may state the following:%
\begin{eqnarray*}
{n \brace k}_{\!\!K\left( r_{1},\ldots ,r_{p}\right) } &=&0,\text{ \  \  \
}n<r_{1}+\cdots +r_{p}\text{ or }k<p, \\
{n \brace k}_{\!\!K\left( r_{1},\ldots ,r_{p}\right) } &=&
{n \brace k}_{\!\!r_{1}+\ldots +r_{p}}\text{ \ if }r_{1},\ldots ,r_{p}\in \left \{
0,1\right \} , \\
{n \brace k}_{\!\!K\left( r_{1},\ldots ,r_{p}\right) } &=&
{n \brace k}_{\!\!K\left( r_{\sigma \left( 1\right) },\ldots ,r_{\sigma \left(
p\right) }\right) }\text{ for all permutations }\sigma \text{ on the set }%
\left \{ 1,\ldots ,p\right \} .
\end{eqnarray*}%
Furthermore, for $r_{1}\leq r_{2}\leq \cdots \leq r_{p},$ the coefficient $%
\alpha _{k}\left( K_{n,\mathbf{r}_{p}}\right) $ represents the number of
ways of partitioning the set  of $%
n+\left \vert \mathbf{r}_{p}\right \vert $ vertices of $K_{n,\mathbf{r}%
_{p}} $ into $k$ independent sets, and by the definition of the graph $K_{n,%
\mathbf{r}_{p}},$ any two elements $x$ of the $i$-th block of the subgraph $%
K_{r_{1},\ldots ,r_{p}}$ and $y$ of the $j$-th block of $K_{r_{1},\ldots
,r_{p}},$ with $i\neq j,$ can't be in the same subset. So, this number is
exactly ${n+r_{1}+\cdots
+r_{p} \brace k}_{\!\!K\left( r_{1},\ldots ,r_{p}\right) }.\ $Then, we may state that%
\begin{equation*}
\alpha _{k}\left( K_{n,\mathbf{r}_{p}}\right) ={n+r_{1}+\cdots
+r_{p} \brace k}_{\!\!K\left( r_{1},\ldots ,r_{p}\right) }:={n+\left
\vert \mathbf{r}_{p}\right \vert  \brace k}_{\!\!K\left( \mathbf{r}_{p}\right) }.
\end{equation*}%

\noindent Set $ H:=K_{r_{1},\ldots ,r_{p}}$ in Theorem \ref{TT1} to obtain:

\begin{corollary}
\label{C2}Let $n,\ k$ be integers such that $p\leq k\leq n$ and $n\geq
\left \vert \mathbf{r}_{p}\right \vert .$ We have%
\begin{equation*}
{n \brace k}_{\!\!K\left( \mathbf{r}_{p}\right) }=\sum_{j=p}^{k}{s+j \brace k}_{\!\!j}{n-s \brace j}_{\!\!K\left( \mathbf{r}_{p}\right) },\  \
0\leq s\leq n-\left \vert \mathbf{r}_{p}\right \vert .
\end{equation*}%
For $s=1$ we conclude that these numbers obey to the recurrence relation:%
\begin{equation*}
{n \brace k}_{\!\!K\left( \mathbf{r}_{p}\right) }=k
{n-1 \brace k}_{\!\!K\left( \mathbf{r}_{p}\right) }+{n-1 \brace k-1}_{\!\!K\left(
\mathbf{r}_{p}\right) },\  \  \ n\geq k\geq 1.
\end{equation*}%
and for $s=n-\left \vert \mathbf{r}_{p}\right \vert $ in Corollary \ref{C2},
we get%
\begin{equation*}
{n \brace k}_{\!\!K\left( \mathbf{r}_{p}\right) }=\sum_{j=p}^{k}
{n-\left \vert \mathbf{r}_{p}\right \vert +j \brace k}_{\!\!j}
{\left \vert \mathbf{r}_{p}\right \vert  \brace j}_{\!\!K\left( \mathbf{r}_{p}\right) }.
\end{equation*}
\end{corollary}

\noindent In particular, for $s=r_{1}=\cdots =r_{p}=1$ in Corollary \ref{C2}
we obtain the known identity:%
\begin{equation*}
{n \brace k}_{\!\!p}=k{n-1 \brace k}_{\!\!p}+{n-1 \brace k-1}_{\!\!p}.
\end{equation*}

\begin{remark}
By setting $s=1\ $and $k=p$ or $p+1$ in Corollary \ref{C2} and upon using the identities given in \cite{zou} (see also \cite[Lemma 4.4.2]{don}) by
\begin{eqnarray*}
{\left \vert \mathbf{r}_{p}\right\vert \brace p}_{\!\!K\left( \mathbf{r}_{p}\right) }=1, \ \ \ {\left
\vert \mathbf{r}_{p}\right \vert \brace p+1}_{\!\!K\left( \mathbf{r}%
_{p}\right) }=\sum_{j=1}^{p}2^{r_{j}-1}-p,
\end{eqnarray*}
we obtain the two initial values of these numbers%
\begin{eqnarray*}
{n \brace p}_{\!\!K\left( \mathbf{r}_{p}\right) } &=&
{n-\left \vert \mathbf{r}_{p}\right \vert +p \brace p}_{\!\!p}, \\
{n \brace p+1}_{\!\!K\left( \mathbf{r}_{p}\right) } &=&
{n-\left \vert \mathbf{r}_{p}\right \vert +p \brace p+1}_{\!\!p}+
{n-\left \vert \mathbf{r}_{p}\right \vert +p+1 \brace p+1}_{\!\!p+1}\left(
\sum_{j=1}^{p}2^{r_{j}-1}-p\right) .
\end{eqnarray*}
\end{remark}

\begin{proposition}
Let
\begin{equation*}
B\left( \lambda ;K_{n,\mathbf{r}_{p}}\right) :=\underset{k\geq 0}{\sum }%
{n+\left \vert \mathbf{r}_{p}\right \vert  \brace k}_{\!\!K\left(
\mathbf{r}_{p}\right) }\lambda ^{k}.
\end{equation*}%
Then, we have%
\begin{eqnarray*}
 B\left( \lambda ;K_{n,\mathbf{r}_{p}}\right)
&=&\lambda\exp \left( -\lambda \right) \frac{d}{d\lambda }\left( \exp \left( \lambda \right) B\left(
\lambda ;K_{n-1,\mathbf{r}_{p}}\right) \right) ,\  \ n\geq 1,\\
B\left( \lambda ;K_{0,\mathbf{r}_{p}}\right) &=&B_{r_{1}}\left( \lambda
\right) \cdots B_{r_{p}}\left( \lambda \right),
\end{eqnarray*}%
where $B_{n}\left(\lambda \right) =\sum_{j=0}^{n} {n \brace j}\lambda ^{j}$ is the classical Bell polynomial.
\end{proposition}

\begin{proof}
From Corollary \ref{C2} we have%
\begin{equation*}
\underset{k\geq 1}{\sum }{n+\left \vert \mathbf{r}_{p}\right \vert \brace k}_{\!\!K\left( \mathbf{r}_{p}\right)
}\lambda^{k}=\underset{k\geq 1}{\sum }k{n+\left \vert \mathbf{r}_{p}\right \vert-1 \brace k}_{\!\!K\left( \mathbf{r}%
_{p}\right) }\lambda^{k}+\underset{k\geq 1}{\sum }{n+\left \vert \mathbf{r}_{p}\right \vert-1 \brace k-1}_{\!\!K\left(
\mathbf{r}_{p}\right) }\lambda^{k}
\end{equation*}%
which gives the first identity. The second one follows from \cite[Lemma 4.4.1]{don}.
\end{proof}

\begin{corollary}
For $n\geq \left \vert \mathbf{r}_{p}\right \vert ,$ the numbers
${n \brace k}_{\!\!K\left( \mathbf{r}_{p}\right) },\ k=p,p+1,\ldots ,n,$ are
log-concave.
\end{corollary}

\begin{proof}
Apply Role's Theorem on the function $f_{n}\left( \lambda \right) :=\exp
\left( \lambda \right) B\left( \lambda ;K_{n,\mathbf{r}_{p}}\right) $ and
use the fact that the roots of the polynomial $B_{n}\left( \lambda \right) $
are real non-positive (see for example \cite {wan}) to conclude (by induction on $n$) that the polynomial $%
B\left( \lambda ;K_{n,\mathbf{r}_{p}}\right) $ has only real non-positive
roots. After that, apply Newton's inequality \cite[p. 52]{har} to complete
the proof.
\end{proof}
\noindent Similarly to Corollary \ref{C3}, Proposition \ref{P1} states that we have
\begin{corollary}
For $0\leq k\leq n$ let
\begin{equation*}
W\left( n,k\right) ={n+\left \vert \mathbf{r}_{p}\right \vert
\brace k+\left \vert \mathbf{r}_{p}\right \vert }_{\!\!K\left( \mathbf{r}_{p}\right) }.
\end{equation*}%
Then, the sequence $\left( W\left( n,k\right) ,\ 0\leq k\leq n\right) $ is a P\'{o}lya frequency sequence.
\end{corollary}

\section{Application to the graph $T_{r_{_{1}}}\cup \cdots \cup T_{r_{_{p}}}\cup O_{n}$}

Let $p\geq 1$ be an integer. To give a third application of the chromatic
polynomials, we consider in this section the graph $T_{n,\mathbf{r}%
_{p}}=T_{r_{_{1}}}\cup \cdots \cup T_{r_{_{p}}}\cup O_{n}.$ Here%
\begin{equation*}
\chi \left( T_{n,\mathbf{r}_{p}}\right) =\max \left( \chi \left(
T_{r_{_{1}}}\right) ,\ldots ,\chi \left( T_{r_{_{p}}}\right) ,\chi \left(
O_{n}\right) \right) =\min \left( r_{1}\cdots r_{p},2\right) .
\end{equation*}%
Similarly to the above numbers, we present here a new class of the Stirling
numbers having also triangular recurrence relation. Then, we may start by
giving the following definition.

\begin{definition}
Let $R_{1},\ldots ,R_{p}$ be subsets of the set $\left[ n\right] $ with $%
\left \vert R_{i}\right \vert =r_{i}$ and $R_{i}\cap R_{j}=\varnothing $ for
all $i,j=1,\ldots ,p,$ $i\neq j.$ The $T\left( r_{1},\ldots ,r_{p}\right) $%
-Stirling number of the second kind, denoted by $
{n \brace k}_{\!\!T\left( r_{1},\ldots ,r_{p}\right) },$ counts the number of
partitions of the set $\left[ n\right] $ into $k$ non-empty subsets such
that the minimum element of $R_{i}$ can't be in the same subset with any
other element of $R_{i},$ $i=1,\ldots ,p.$
\end{definition}

\noindent From this definition, we may state the following:%
\begin{eqnarray*}
{n \brace k}_{\!\!T\left( r_{1},\ldots ,r_{p}\right) } &=&0,\text{ \  \  \
}n<r_{1}+\cdots +r_{p}\text{ or }k<\min \left( r_{1}\cdots r_{p},2\right) ,
\\
{n \brace k}_{\!\!T\left( r_{1},\ldots ,r_{p}\right) } &=&
{n \brace k}\text{ \ if }r_{1},\ldots ,r_{p}\in \left \{ 0,1\right \} , \\
{n \brace k}_{\!\!T\left( r_{1},\ldots ,r_{p}\right) } &=&
{n \brace k}_{\!\!K\left( r_{\sigma \left( 1\right) },\ldots ,r_{\sigma \left(
p\right) }\right) }\text{ for all permutations }\sigma \text{ on the set }%
\left \{ 1,\ldots ,p\right \} .
\end{eqnarray*}%
Furthermore, since the trees of same order have the same chromatic polynomial,
it is sufficient to apply such results on the star graphs. So, choice $T_{i}$
to be the star graph with $r_{i}$ vertices ($T_{i}=K_{1,r_{i}-1}$ ) with the universal vertex is the minimum element of $%
R_{i}$ (the set-vertex of $T_{i}$). Hence, the coefficient $\alpha _{k}\left(
T_{n,\mathbf{r}_{p}}\right) $ represents the number of ways of partitioning
the set of $n+\left \vert \mathbf{r}_{p}\right \vert $ vertices of $T_{n,\mathbf{r}_{p}}$ into $k$
independent sets, and by the definition of the graph $T_{n,\mathbf{r}_{p}},$
the universal vertex of $T_{r_{i}}$ ($i=1,\ldots ,p$) can't be in the same
independent set with any other vertex of $T_{r_{i}}$. So, this number is
exactly ${n+r_{1}+\cdots+r_{p} \brace k}_{\!\!T\left( r_{1},\ldots ,r_{p}\right) }.$ \\ Then, we may state that%
\begin{equation*}
\alpha _{k}\left( T_{n,\mathbf{r}_{p}}\right) ={n+r_{1}+\cdots
+r_{p} \brace k}_{\!\!T\left( r_{1},\ldots ,r_{p}\right) }:={n+\left
\vert \mathbf{r}_{p}\right \vert \brace k}_{\!\!T\left( \mathbf{r}_{p}\right) }.
\end{equation*}%

\begin{theorem}
For $1\leq r_{1}\leq r_{2}\leq \cdots \leq r_{p},$ the polynomial $%
z^{n+p}\left( z-1\right) ^{\left \vert \mathbf{r}_{p}\right \vert -p}$ can
be written in the basis $\left \{ \left( \lambda \right) _{k},\ k=0,1,\ldots,n+\left \vert \mathbf{r}_{p}\right \vert
\right \} $ as follows:%
\begin{equation*}
z^{n+p}\left( z-1\right) ^{\left \vert \mathbf{r}_{p}\right \vert
-p}=\sum_{k=0}^{n+\left \vert \mathbf{r}_{p}\right \vert }
{n+\left \vert \mathbf{r}_{p}\right \vert \brace k}_{\!\!T\left( \mathbf{r}%
_{p}\right) }\left( z\right) _{k},
\end{equation*}%
Furthermore, we have%
\begin{equation*}
{n+\left \vert \mathbf{r}_{p}\right \vert \brace k}_{\!\!T\left(
\mathbf{r}_{p}\right) }={n+\left \vert \mathbf{r}_{p}\right
\vert \brace k}_{\!\!T\left( \left \vert \mathbf{r}_{p}\right \vert -p+1\right) }.
\end{equation*}
\end{theorem}

\begin{proof}
From \cite[Sec. 1.2]{don}, the chromatic polynomial of the graph $T_{n,%
\mathbf{r}_{p}}$ is
\begin{eqnarray*}
P\left( T_{n,\mathbf{r}_{p}},z\right) &=&P\left( O_{n},z\right) P\left(
T_{r_{1}},z\right) \cdots P\left( T_{r_{p}},z\right) \\&=&\left( z\right)
^{n+p}\left( z-1\right) ^{\left \vert \mathbf{r}_{p}\right \vert
-p}\\&=&\sum_{i=0}^{n+\left \vert \mathbf{r}_{p}\right \vert }
{n+\left \vert \mathbf{r}_{p}\right \vert \brace i}_{\!\!T\left( \mathbf{r}%
_{p}\right) }\left( z\right) _{i}.
\end{eqnarray*}%
This gives the first result. In particular, we have
\begin{equation*}
z^{n+1}\left(z-1\right) ^{r-1}=\sum_{k=0}^{n}{n+r \brace k}_{\!\!T\left( r\right)
}\left( z\right) _{k},
\end{equation*}%
and by combining this result with the first one we obtain the second result.
\end{proof}

\noindent For the choice $H:=T_{r_{_{1}}}\cup \cdots \cup
T_{r_{_{p}}}$ in Theorem \ref{TT1} we get:

\begin{corollary}
Let $n,\ k$ be integers such that $2\leq k\leq n$ and $n\geq \left \vert
\mathbf{r}_{p}\right \vert .$ Then, for $1\leq r_{1}\leq \cdots
\leq r_{p},$ we have%
\begin{equation*}
{n \brace k}_{\!\!T\left( \mathbf{r}_{p}\right)
}=\sum_{j=2}^{k}{s+j \brace k}_{\!\!j}
{n-s \brace j}_{\!\!T\left( \mathbf{r}_{p}\right) },\  \ 0\leq s\leq n-\left \vert
\mathbf{r}_{p}\right \vert .
\end{equation*}%
In particular, for $s=1,$ these numbers obey to the following triangular
recurrence relation:%
\begin{equation*}
{n \brace k}_{\!\!T\left( \mathbf{r}_{p}\right) }=
{n-1 \brace k-1}_{\!\!T\left( \mathbf{r}_{p}\right) }+k{n-1 \brace k}_{\!\!T\left(
\mathbf{r}_{p}\right) },
\end{equation*}%
and for $s=n-\left \vert \mathbf{r}_{p}\right \vert $ we get%
\begin{equation*}
{n \brace k}_{\!\!T\left( \mathbf{r}_{p}\right)
}=\sum_{j=2}^{k}{n-\left \vert \mathbf{r}_{p}\right \vert
+j \brace k}_{\!\!j}{\left \vert \mathbf{r}_{p}\right \vert
\brace j}_{\!\!T\left( \mathbf{r}_{p}\right) }.
\end{equation*}
\end{corollary}

\noindent For $i=1,\ldots ,p,$ set $\ s=1$ and $T_{r_{i}}\cup H:=T_{r_{_{1}}}\cup
\cdots \cup T_{r_{_{p}}}\cup O_{n}\ $ in Theorem \ref{TT3} to obtain:

\begin{corollary}
Let $n,\ k$ be integers such that $1\leq k\leq n$ and $n\geq \left \vert
\mathbf{r}_{p}\right \vert .$ Then, for a fixed $i$ \ $(i=1,\ldots ,p)$ with $r_{i}\geq 1,$  we have%
\begin{equation*}
{n+\left \vert \mathbf{r}_{p}\right \vert \brace k}_{\!\!T\left(
\mathbf{r}_{p}\right) } =\left( k-1\right){n+\left \vert \mathbf{r}_{p}\right \vert \brace k}_{\!\!T\left(
\mathbf{r}_{p}-\mathbf{e}_{i}\right) }+{n+\left \vert \mathbf{r}_{p}\right \vert \brace k-1}_{\!\!T\left(
\mathbf{r}_{p}-\mathbf{e}_{i}\right) } ,\  \ n\geq 1.
\end{equation*}
\end{corollary}

\begin{remark}
For $1\leq r_{1}\leq \cdots \leq r_{p},$ then similarly to Remark %
\ref{R1}, we get%
\begin{equation*}
{n+\left \vert \mathbf{r}_{p}\right \vert \brace k}_{\!\!T\left(
\mathbf{r}_{p}\right) }=\frac{1}{k!}\sum_{j=0}^{k}\left( -1\right) ^{k-j}%
\binom{k}{j}j^{n+p}\left( j-1\right)
^{\left \vert \mathbf{r}_{p}\right \vert -p},
\end{equation*}%
and since $\chi \left( T_{n,\mathbf{r}_{p}}\right) =2,$ when $r_{1}\cdots r_{p}>1$, then the two initial values of these numbers are given by
\begin{equation*}
{n+\left \vert \mathbf{r}_{p}\right \vert \brace 2}_{\!\!T\left(
\mathbf{r}_{p}\right) }=2 ^{n+p-1}, \ \ {n+\left \vert \mathbf{r}_{p}\right \vert \brace 3}_{\!\!T\left(
\mathbf{r}_{p}\right) }=3 ^{n+p-1}\times2^{\left \vert \mathbf{r}_{p}\right \vert -p-1}-2 ^{n+p-1},
\end{equation*}
and for $r_{1}=\cdots =r_{p}=1$ we get
\begin{equation*}
{n \brace k}_{\!\!T\left(\mathbf{r}_{p}\right) }={n \brace k}.
\end{equation*}%
\end{remark}
\noindent Similarly to Corollary \ref{C3}, we have:
\begin{corollary}
For $0\leq k\leq n$ let
\begin{equation*}
V\left( n,k\right) ={n+\left \vert \mathbf{r}_{p}\right \vert
\brace k+\left \vert \mathbf{r}_{p}\right \vert }_{\!\!T\left( \mathbf{r}_{p}\right)
}.
\end{equation*}%
Then, the sequence $\left( V\left( n,k\right) ,\ 0\leq k\leq n\right) $ is
log-concave and P\'{o}lya frequency sequence, and, the sequence of
polynomials $\left( V_{n}\left( q\right) ,n\geq 0\right) $ defined by
\begin{equation*}
V_{n}\left( q\right) =\overset{n}{\underset{k=0}{\sum }}{n+\left \vert \mathbf{r}_{p}\right \vert
\brace k+\left \vert \mathbf{r}_{p}\right \vert }_{\!\!T\left( \mathbf{r}_{p}\right)
}q^{k}
\end{equation*}%
is a $q$-log-convex sequence.
\end{corollary}

\end{document}